\documentclass[english]{smfart}
\usepackage{amsmath, amssymb, amsthm, amscd,smfenum,xspace,mathrsfs,url,upref,verbatim}
\usepackage[utf8]{inputenc} 
\usepackage[T1]{fontenc}
\usepackage{dsfont}

\usepackage{mathabx}
\usepackage{stmaryrd}

\usepackage[all,cmtip]{xy}
\usepackage{xfrac} 

\usepackage{tikz-cd}
\usepackage{tikz}

\let\hat=\widehat
\let\tilde=\widetilde
\numberwithin{equation}{subsection}

\newtheorem{theorem}{Theorem}

\newtheorem{question}{Question} 
\newtheorem{proposition}[equation]{Proposition}
\newtheorem{lemme}[equation]{Lemma}

\theoremstyle{remark}

\DeclareMathOperator{\Image}{Im}

\DeclareMathOperator{\Max}{Max}
\DeclareMathOperator{\id}{id}

\DeclareMathOperator{\Spec}{Spec}

\def\cartesien{\ar@{}[rd]|{\square}}

\DeclareMathOperator{\GL}{GL}

\DeclareMathOperator{\Gal}{Gal}

\DeclareMathOperator{\sh}{sh}

\DeclareMathOperator{\Supp}{Supp}

\DeclareMathOperator{\rg}{rg}
\DeclareMathOperator{\Sl}{Sl}

\DeclareMathOperator{\Sh}{Sh}

\DeclareMathOperator{\nb}{nb}

\DeclareMathOperator{\trdeg}{trdeg}
\usepackage{enumerate}

\author[J.-B.~Teyssier]{Jean-Baptiste Teyssier}
\curraddr{Freie Universität Berlin, Mathematisches Institut, Arnimallee 3, 14195 Berlin, Germany}
\email{teyssier@zedat.fu-berlin.de}

\usepackage[english,french]{babel}
\title[Nearby slopes and boundedness for $\ell$-adic sheaves in char. $>0$]{Nearby slopes and boundedness for $\ell$-adic sheaves in positive characteristic}
\begin{document}    
\maketitle
\section*{Introduction}
Let $S$ be a strictly henselian trait of equal characteristic $p>0$. As usual, $s$ denotes the closed point of $S$, $k$ its residue field, $\eta=\Spec K$ the generic point of $S$, $\overline{K}$ an algebraic closure of $K$ and $\overline{\eta}=\Spec \overline{K}$. Let $f: X\longrightarrow S$ be a morphism of finite type, $\ell\neq p$ a prime number, $\mathcal{F}$  an object of the derived category $D^{b}_c(X_{\eta}, \overline{\mathds{Q}}_{\ell})$ of $\ell$-adic complexes with bounded and constructible cohomology. \\ \indent Let $\psi_f^{t}: D^{b}_{c}(X_{\eta},  \overline{\mathds{Q}}_{\ell})\longrightarrow  D^{b}_{c}(X_{s},  \overline{\mathds{Q}}_{\ell})$ be the moderate nearby cycle functor. We say that $r\in \mathds{R}_{\geq 0}$ is a \textit{nearby slope of $\mathcal{F}$ associated to $f$} if one can find $N\in \Sh_c(\eta, \overline{\mathds{Q}}_{\ell})$ with slope $r$ such that $\psi_{f}^{t}(\mathcal{F}\otimes f^{\ast}N)  \neq 0$. We denote by $\Sl^{\nb}_{f}(\mathcal{F})$ the set of nearby slopes of $\mathcal{F}$ associated to $f$.\\ \indent
The main result of \cite{NbTey} is a boundedness theorem for the set of nearby slopes of a complex holonomic $\mathcal{D}$-module. The goal of the present (mostly programmatic) paper is to give some motivation for an analogue of this theorem for $\ell$-adic sheaves in positive chara\-cteristic. \\ \indent
For complex holonomic $\mathcal{D}$-modules, regularity is preserved by push-forward. On the other hand, for a morphism $C^{\prime}\longrightarrow C$ between smooth curves over $k$, a tame constructible sheaf on $C^{\prime}$  may acquire wild ramification by push-forward. If $0\in C$ is a closed point, the failure of $C^{\prime}\longrightarrow C$ to preserve tameness above 0 is accounted for by means of the ramification filtration on the absolute Galois group of the function field of the strict henselianization $C_{0}^{\sh}$ of $C$ at $0$. Moreover,
the Swan conductor at $0$ measures to which extent an $\ell$-adic constructible sheaf on $C$ fails to be tame at $0$. \\ \indent
In higher dimension, both these measures of wild ramification (for a morphism and for a sheaf) are missing in a form that would give a precise meaning to the following question raised in \cite{maildu22mars}  
\begin{question}\label{laquestion}
Let $g: V_1\longrightarrow V_2$ be a morphism between schemes of finite type over $k$, and $\mathcal{G}\in D^{b}_c(V_1, \overline{\mathds{Q}}_{\ell})$. Can one bound the wild ramification of $Rg_{\ast}\mathcal{G}$ in terms of the wild ramification of $\mathcal{G}$ and the wild ramification of $g_{|\Supp \mathcal{G}}$?
\end{question} 
Note that in an earlier formulation, "wild ramification of $g_{|\Supp \mathcal{G}}$" was replaced by "wild ramification of $g$", which cannot hold due to the following example that we owe to Alexander Beilinson: take $f:\mathds{A}^{1}_S\longrightarrow S$, $P\in S[t]$ and $i_P:\{P=0\}\hookrightarrow  \mathds{A}^{1}_S$. Then $i_{P\ast}\overline{\mathds{Q}}_{\ell}$ is tame but $f_{\ast}(i_{P\ast}\overline{\mathds{Q}}_{\ell})$ has arbitrary big wild ramification as $P$ runs through the set of Eisenstein polynomials.\\ \indent
If $f:X\longrightarrow S$ is proper, proposition \ref{control} shows that   $\Sl^{\nb}_{f}(\mathcal{F})$ controls the slopes of $H^{i}(X_{\overline{\eta}}, \mathcal{F})$ for every $i\geq 0$. It is thus tempting to take for "wild ramification of $\mathcal{G}$" the nearby slopes of $\mathcal{G}$.\\ \indent
So Question \ref{laquestion} leads to the question of bounding nearby slopes of  constructible $\ell$-adic sheaves. Note that this question was raised imprudently in \cite{NbTey}. It has a negative answer as stated in $loc.$ $it.$ since already the constant sheaf $\overline{\mathds{Q}}_{\ell}$ has arbitrary big nearby slopes. 
This is actually good news since for curves, these nearby slopes keep track of the aforementioned ramification filtration\footnote{see \ref{deuxiemefact} \eqref{3} for a precise statement.}. Hence, one can use them in higher dimension to quantify the wild ramification of a morphism and in Question \ref{laquestion} take for "wild ramification of $g_{|\Supp \mathcal{G}}$" the nearby slopes of $\overline{\mathds{Q}}_{\ell}$ on $\Supp \mathcal{G}$ associated with $g_{|\Supp \mathcal{G}}$ (at least when $V_2$ is a curve). \\ \indent
To get a good boundedness statement, one has to correct the nearby slopes associated with a morphism by taking into account the maximal nearby slope of $\overline{\mathds{Q}}_{\ell}$ associated with the same morphism. That such a maximal slope exists in general is a consequence of the following
\begin{theorem}\label{theofini}
Let $f: X\longrightarrow S$ be a morphism of finite type and $\mathcal{F}\in D^{b}_c(X_{\eta}, \overline{\mathds{Q}}_{\ell})$. The set $\Sl^{\nb}_{f}(\mathcal{F})$ is finite.
\end{theorem}
The proof of this theorem follows an argument due to Deligne \cite[Th. finitude 3.7]{SGA4undemi}. For a $\mathcal{D}$-module version, let us refer to \cite{DeligneLettreMalgrange}. Thus, $\Max \Sl^{\nb}_{f}(\overline{\mathds{Q}}_{\ell})$ makes sense if $\Sl^{\nb}_{f}(\overline{\mathds{Q}}_{\ell})$ is not empty. Otherwise, we set $\Max \Sl^{\nb}_{f}(\overline{\mathds{Q}}_{\ell})=+\infty$. Proposition \ref{bornedim1}  suggests and gives a positive answer to the following question for smooth curves
\begin{question}\label{laquestion1}
Let $V/k$ be a scheme of finite type  and $\mathcal{F}\in D^{b}_c(V, \overline{\mathds{Q}}_{\ell})$. Is it true that the following set
\begin{equation}\label{apoids}
\{r/(1+\Max \Sl^{\nb}_{f}(\overline{\mathds{Q}}_{\ell})), \text{ for } r\in \Sl^{\nb}_{f}(\mathcal{F}) \text{ and } f\in \mathcal{O}_V\}
\end{equation}
is bounded?
\end{question} 
Let us explain what $\Sl^{\nb}_{f}(\mathcal{F})$ means in this global setting. A function $f\in \Gamma(U, \mathcal{O}_V)$ reads as $f:U\longrightarrow \mathds{A}^{1}_{k}$. If $S$ is the strict henselianization of $\mathds{A}^{1}_{k}$ at a geometric point over the origin, we set $\Sl^{\nb}_{f}(\mathcal{F}):=\Sl^{\nb}_{f_S}(\mathcal{F}_{U_S})$ where the subscripts are synonyms of pull-back.\\ \indent
For smooth curves, the main point of the proof of boundedness is the concavity of Herbrand $\varphi$ functions. In case $f$ has generalized semi-stable reduction (see \ref{gssr}), the above weighted slopes are the usual nearby slopes. This is  the following \newpage
\begin{theorem}\label{mainth}
Suppose that  $f:X\longrightarrow S$ has generalized semi-stable reduction. Then we have $\Sl^{\nb}_{f}(\overline{\mathds{Q}}_{\ell})=\{0\}$. 
\end{theorem}
We owe the proof of this theorem to Joseph Ayoub. For the vanishing of $\mathcal{H}^{0}\psi^{t}_f$, we also give an earlier argument based on the geometric connectivity of the connected components of the moderate Milnor fibers in case of generalized semi-stable reduction. \\ \indent
As a possible application of a boundedness theorem in the arithmetic setting, let us remark that for every compactification $j: V\longrightarrow \overline{V}$, one could define a separated decreasing $\mathds{R}_{\geq 0}$-filtration on $\pi_1(V)$ by looking for each $r\in \mathds{R}_{\geq 0}$ at the category of $\ell$-adic local systems $L$ on $V$ such that the weighted slopes \eqref{apoids} of $j_{!}L$ are $\leq r$. 
\\ \indent
Let us also remark that on a smooth curve $C$, the tameness of $\mathcal{F}\in \Sh_c(C, \overline{\mathds{Q}}_{\ell})$ at $0\in C$ is characterized by $\Sl_{f}^{\nb}(\mathcal{F})\subset [0,\Max \Sl_{f}^{\nb}(\overline{\mathds{Q}}_{\ell})]$ for every $f\in \mathcal{O}_C$ vanishing only at $0$. This suggests a notion of tame complex in any dimension that may be of interest.
 \\ \indent

I thank Joseph Ayoub for his willingness to know about nearby slopes and for generously explaining me a proof of Theorem \ref{mainth} during a stay in Zurich in May 2015. I also thank Kay Rülling for a useful discussion.
This work has been achieved with the support of Freie Universität/Hebrew University of Jerusalem joint post-doctoral program. I thank Hélène Esnault and Yakov Varshavsky for their support.

\section{Notations}\label{notation}
\subsection{} For a general reference on wild ramification in dimension 1, let us mention \cite{CL}. Let $\eta_t$ be the point of $S$ corresponding to the tamely ramified closure $K_t$ of $K$ in $\overline{K}$ and $P_K:=\Gal(\overline{K}/K_t)$ the wild ramification group of $K$. We denote by $(G^{r}_K)_{r\in \mathds{R}_{\geq 0}}$ the \textit{upper-numbering ramification filtration on $G_K$} and define 
$$
G_K^{r+}:=\overline{\displaystyle{\bigcup}_{r^{\prime}>r}G^{r^{\prime}}_K}  $$
If $L/K$ is a finite extension, we denote by $S_L$ the normalization of $S$ in $L$ and $v_L$ the valuation on $L$ associated with the maximal ideal of $S_L$. \\ \indent
If moreover $L/K$ is separable, we denote by $q:G_K\longrightarrow G_K/G_L$  the quotient morphism and define a decreasing separated $\mathds{R}_{\geq 0}$-filtration on the set $G_K/G_L$ by $(G_K/G_L)^{r}:=q(G_K^{r})$. We also define $(G_K/G_L)^{r+}:=q(G_K^{r+})$. \\ \indent
In case $L/K$ is Galois, this filtration is the upper numbering ramification filtration on $\Gal(L/K)$. If $L/K$ is non separable trivial, the \textit{jumps} of $L/K$ are the $r\in \mathds{R}_{\geq 0}$ such that $(G_K/G_L)^{r+}\subsetneq (G_K/G_L)^{r}$. If $L/K$ is trivial, we say by convention that $0$ is the only jump of $\Gal(L/K)$.  
\subsection{} 
For $M\in D^{b}_c(\eta, \overline{\mathds{Q}}_{\ell})$, we denote by $\Sl(M)\subset \mathds{R}_{\geq 0}$ the set of \textit{slopes} of $M$ as defined in \cite[Ch 1]{KatzGKM}. We view $M$ in an equivalent way as a continuous representation of $G_K$. 
\subsection{}
Let  $f:X\longrightarrow S$ be a morphism of finite type and  $\mathcal{F}\in D^{b}_c(X_{\eta}, \overline{\mathds{Q}}_{\ell})$. Consider the following diagram with cartesian squares
$$
\xymatrix{
 X_s \ar[r]^-{i} \ar[d]&     X        \ar[d]_-{f}              & \ar[l]_-{\overline{j}}   X_{\overline{\eta}}  \ar[d] \\
s        \ar[r]       &   S  &  \ar[l]  \overline{\eta}
}
$$
Following \cite[XIII]{SGA7-2}, we define the \textit{nearby cycles of  $\mathcal{F}$} as
$$
\psi_f \mathcal{F} := i^{\ast}R\overline{j}_{\ast}\overline{j}^{\ast}\mathcal{F}
$$
By \cite[Th. finitude 3.2]{SGA4undemi}, the complex $\psi_f \mathcal{F}$ is an object of $D^{b}_c(X_{s}, \overline{\mathds{Q}}_{\ell})$ endowed with a continuous $G_K$-action. Define $X_{t}:=X\times_{S} \eta_t$ and $j_t: X_t\longrightarrow X$ the projection. Following \cite[I.2]{SGA7-1}, we define the \textit{moderate nearby cycles of  $\mathcal{F}$} as
$$
\psi_f^{t} \mathcal{F} := i^{\ast}Rj_{t\ast}j^{\ast}_t\mathcal{F}
$$
It is a complex in $D^{b}_c(X_{s}, \overline{\mathds{Q}}_{\ell})$ endowed with a continuous $G/P_K$-action. Since $P_K$ is a pro-$p$ group, we have a canonical identification 
$$
\psi_f^{t} \mathcal{F} \simeq (\psi_f \mathcal{F})^{P_K}
$$
Note that by proper base change \cite[XII]{SGA4-3}, $\psi_f^{t}$ and $\psi_f$ are compatible with proper push-forward.
\subsection{}\label{gssr} By a \textit{generalized semi-stable reduction} morphism, we mean a morphism $f:X\longrightarrow S$ of finite type such that etale locally on $X$, $f$ has the form
$$
S[x_1,\dots, x_n]/(\pi-x_1^{a_1}\cdots x_m^{a_m})\longrightarrow S
$$
where $\pi$ is a uniformizer of $S$ and where the $a_i\in \mathds{N}^{\ast}$ are prime to $p$. 
\subsection{}  If $X$ is a scheme, $x\in X$ and if $\overline{x}$ is a geometric point of $X$ lying over $X$, we denote by $X_{x}^{\sh}$ the strict henselization of $X$ at $x$.

\section{Nearby slopes in dimension one}
\subsection{} We show here that nearby slopes associated with the identity morphism are the usual slopes as in \cite[Ch 1]{KatzGKM}.
\begin{lemme}\label{toutptilemme}
For every $M\in \Sh_c(\eta,\overline{\mathds{Q}}_{\ell})$, we have
$$
\Sl_{\id}^{\nb}(M)=\Sl(M)
$$
\end{lemme}
\begin{proof}
We first remark that $\psi^{t}_{\id}$ is just the "invariant under $P$" functor. Suppose that $r\in \Sl(M)$. Then $M$ has a non zero quotient $N$ purely of slope $r$. The dual $N^{\vee}$ has pure slope $r$. Since $N$ is non zero, the canonical map
$$
\xymatrix{
N\otimes N^{\vee}\ar[r] &  \overline{\mathds{Q}}_{\ell}
}
$$
is surjective. Since taking $P$-invariants is exact, we obtain that the maps in 
$$
\xymatrix{
(M\otimes N^{\vee})^{P}\ar[r] &  (N\otimes N^{\vee})^{P}\ar[r] &  \overline{\mathds{Q}}_{\ell}
}
$$
are surjective. Hence $(M\otimes N^{\vee})^{P}\neq 0$, so $r\in \Sl_{\id}^{\nb}(M)$.\\ \indent
If $r$ is not a slope of $M$, then for any $N$ of slope $r$, the slopes of $M\otimes N$ are non zero. This is equivalent to $(M\otimes N)^{P}=0$. 
\end{proof}
 We deduce the following
\begin{lemme}\label{deuxiemefact}
Let $f: X\longrightarrow S$  be a finite morphism with $X$ local and $\mathcal{F}\in \Sh_c(X_\eta, \overline{\mathds{Q}}_{\ell})$.  
\begin{enumerate}
\item\label{1} $\Sl^{\nb}_{f}(\mathcal{F})=\Sl(f_{\ast}\mathcal{F})$.
\item\label{2}  Suppose that $X$ is regular connected and let $L/K$ be the extension of function fields induced by $f$. Suppose that $L/K$ is separable. Then $\Max \Sl^{\nb}_{f}(\overline{\mathds{Q}}_{\ell})$ is the highest jump in the ramification filtration on $G_K/G_L$.
\item\label{3} Suppose further in \eqref{2} that $L/K$ is Galois and set $G:=\Gal(L/K)$. Then $\Sl^{\nb}_{f}(\overline{\mathds{Q}}_{\ell})$ is the union of $\{0\}$ with the set of jumps in the ramification filtration on $G$.
\end{enumerate}
\end{lemme}
\begin{proof}
Point \eqref{1} comes from \ref{toutptilemme} and the compatibility of $\psi_f^{t}$  with proper push-forward. \\ \indent
From point $\eqref{1}$ and $f_{\ast}\overline{\mathds{Q}}_{\ell}\simeq \overline{\mathds{Q}}_{\ell}[G_K/G_L]$, we deduce
$$
\Sl^{\nb}_{f}(\overline{\mathds{Q}}_{\ell})=\Sl(\overline{\mathds{Q}}_{\ell}[G_K/G_L])
$$
If $L/K$ is trivial, \eqref{2} is true by our definition of jumps in that case. If $L/K$ is non trivial, $r_{\max}=\Max \Sl(\overline{\mathds{Q}}_{\ell}[G_K/G_L])$ is characterized by the property that  $G^{r_{\max}}_K$ acts non trivially on $\overline{\mathds{Q}}_{\ell}[G_K/G_L]$ and $G^{r_{\max}+}_K$ acts trivially. On the other hand, the highest jump $r_0$ in the ramification filtration on $G_K/G_L$ is such that $q(G^{r_0}_K)\neq \{G_L\}$ and $q(G^{r_0+}_K)= \{G_L\}$, that is $G^{r_0}_K\nsubset G_L$ and  $G^{r_0+}_K\subset G_L$. The condition
$G^{r_0}_K\nsubset G_L$  ensures that $G^{r_0}_K$ acts non trivially on 
$\overline{\mathds{Q}}_{\ell}[G_K/G_L]$. If $h\in G^{r_0+}_K$, then for every $g\in G_K$
$$
h\cdot (gG_L)=hg G_L=gg^{-1}hg G_L=g G_L
$$
where the last equality comes from the fact that since $G^{r_0+}_K$ is a normal subgroup in $G_K$, we have $g^{-1}hg\in G^{r_0+}_K\subset G_L$. So \eqref{2} is proved. \\ \indent
Let $S$ be the union of $\{0\}$ with the set of jumps in the ramification filtration of $G$. To prove \eqref{3}, we have to prove 
$\Sl(\overline{\mathds{Q}}_{\ell}[G])=S$. If $r\in \mathds{R}_{\geq 
0}$ does not belong to $S$, we can find an open interval $J$ 
containing $r$  such that $G^{r^{\prime}}=G^{r}$ for every 
$r^{\prime}\in J$. In particular, the image of $G^{r^{\prime}}_K$ by 
$G_K\longrightarrow \GL(\overline{\mathds{Q}}_{\ell}[G])$ does not depend on 
$r^{\prime}$ for every $r^{\prime}\in J$. So $r$ is not a slope 
of $\overline{\mathds{Q}}_{\ell}[G]$. \\ \indent
Reciprocally, $\overline{\mathds{Q}}_{\ell}[G]$ contains a copy of the trivial representation, so $0\in \Sl(\overline{\mathds{Q}}_{\ell}[G])$. Let $r\in S\setminus \{0\}$. The projection morphism $G\longrightarrow G/G^{r+}$ induces a surjection of $G_K$-representations
$$
\xymatrix{
  \overline{\mathds{Q}}_{\ell}[G] \ar[r] &   \overline{\mathds{Q}}_{\ell}[G/G^{r+}] \ar[r] & 0
}
$$
So $\Sl(\overline{\mathds{Q}}_{\ell}[G/G^{r+}] )\subset \Sl( \overline{\mathds{Q}}_{\ell}[G])$. Note that $G^{r+}$ acts trivially on $\overline{\mathds{Q}}_{\ell}[G/G^{r+}]$.
By definition $G^{r+}\subsetneq G^{r}$, so $G^{r}$ acts non trivially on $\overline{\mathds{Q}}_{\ell}[G/G^{r+}]$. So $r=\Max \Sl(\overline{\mathds{Q}}_{\ell}[G/G^{r+}])$ and point \eqref{3} is proved.

\end{proof}

\subsection{}\label{drawsomecons}  Let us draw a consequence of \ref{toutptilemme}. We suppose that $f:X\longrightarrow S$ is proper. Let $\mathcal{F}\in D^{b}_c(X_\eta, \overline{\mathds{Q}}_{\ell})$. The $G_K$-module associated to
$R^{k}f_{\ast} \mathcal{F}\in D^{b}_c(\eta, \overline{\mathds{Q}}_{\ell})$ is $H^{k}(X_{\overline{\eta}},\mathcal{F})$. From \ref{toutptilemme}, we deduce
\begin{align*}
\Sl(H^{k}(X_{\overline{\eta}},\mathcal{F}))&=\Sl^{\nb}_{\id}(R^{k}f_{\ast} \mathcal{F})\\
 & \subset \Sl^{\nb}_{\id}(Rf_{\ast} \mathcal{F})
\end{align*}
where the inclusion comes from the fact that taking $P_K$-invariants is exact. For every $N\in \Sh_c(\eta,\overline{\mathds{Q}}_{\ell})$, the projection formula and the compatibility of $\psi_f^{t}$ with proper push-forward gives
\begin{align*}
\psi_{\id}^{t}(Rf_{\ast} \mathcal{F}\otimes N)&\simeq  \psi_{\id}^{t}(Rf_{\ast} (\mathcal{F}\otimes f^{\ast} N)    )  \\
&\simeq Rf_{\ast}\psi_{f}^{t} (\mathcal{F}\otimes f^{\ast} N)  
\end{align*}
Hence we have proved the following
\begin{proposition}\label{control}
Let $f:X\longrightarrow S$ be a proper morphism, and let $\mathcal{F}\in D^{b}_c(X_\eta, \overline{\mathds{Q}}_{\ell})$. For every $i\geq 0$, we have
$$
\Sl(H^{i}(X_{\overline{\eta}},\mathcal{F}))\subset \Sl^{\nb}_{f}(\mathcal{F})$$
\end{proposition}
\subsection{Boundedness}
We first need to see that the upper-numbering filtration is unchanged by purely inseparable base change. This is the following
\begin{lemme}\label{purelyins}
Let $K^{\prime}/K$ be a purely inseparable extension of degree $p^{n}$. Let $L/K$ be finite Galois extension,  $L^{\prime}:=K^{\prime}\otimes_K L$ the associated Galois extension of $K^{\prime}$. Then, the isomorphism
\begin{eqnarray}
\label{isogal0}\Gal(L/K) & \overset{\sim}{\longrightarrow}    &\Gal(L^{\prime}/K^{\prime})\\
g &  \longrightarrow     &  \id\otimes g
\end{eqnarray}
is compatible with the upper-numbering filtration.
\end{lemme}
\begin{proof}
Note that for every $g\in \Gal(L/K)$, $\id\otimes g \in \Gal(L^{\prime}/K^{\prime})$ is determined by the property that its restriction to $L$ is $g$. \\ \indent
Let $\pi$ be a uniformizer of $S$ and $\pi_L$ a uniformizer of $S_L$. We have $K\simeq k((\pi))$ and $L\simeq k((\pi_L ))$. Since $k$ is perfect and since $K^{\prime}/K$  and $L^{\prime}/L$ are purely inseparable of degree $p^{n}$, we have $K^{\prime}=k((\pi^{1/p^{n}} ))$ and $L^{\prime}=k((\pi^{1/p^{n}}_L))$. So $\pi^{1/p^{n}}_L$ is a uniformizer of $S_{L^{\prime}}$. For every $\sigma\in \Gal(L^{\prime}/K^{\prime})$ we have
$$
(\sigma(\pi^{1/p^{n}}_L)-\pi^{1/p^{n}}_L)^{p^{n}}=\sigma_{|L}(\pi_L)-\pi_L
$$
so 
\begin{align*}
v_{L^{\prime}}(\sigma(\pi^{1/p^{n}}_L)-\pi^{1/p^{n}}_L)&=\frac{1}{p^{n}}v_{L^{\prime}}(\sigma_{|L}(\pi_L)-\pi_L)   \\ 
  & =v_L(\sigma_{|L}(\pi_L)-\pi_L) 
\end{align*}
So \eqref{isogal0} commutes with the lower-numbering filtration. Hence, \eqref{isogal0} commutes with the upper-numbering filtration and lemma \ref{purelyins} is proved.
\end{proof}
Boundedness in case of smooth curves over $k$ is a consequence of the following
\begin{proposition}\label{bornedim1}
Let $S_0$ be an henselian trait over $k$, let $\eta_0=\Spec K_0$ be the generic point of $S_0$ and $M\in \Sh_c(\eta_{0},\overline{\mathds{Q}}_{\ell})$.
There exists a constant $C_M\geq 0$ depending only on $M$ such that for every finite morphism $f: S_0\longrightarrow S$, we have
\begin{equation}\label{quantity}
\Sl_f^{\nb}(M)\subset [0,\Max(C_M,\Max\Sl_f^{\nb}(\overline{\mathds{Q}}_{\ell})) ]
\end{equation}
In particular, the quantity
$$
\Max\Sl_f^{\nb}(M)/(1+\Max\Sl_f^{\nb}(\overline{\mathds{Q}}_{\ell}))
$$
is bounded uniformely in $f$.
\end{proposition}
\begin{proof}
By \ref{deuxiemefact} \eqref{1}, we have to bound $\Sl(f_{\ast}M)$ in terms of $\Max\Sl(f_{\ast}\overline{\mathds{Q}}_{\ell})$. 
Using \cite[I 1.10]{KatzGKM}, we can replace $\overline{\mathds{Q}}_{\ell}$ by $\mathds{F}_{\lambda}$, where $\lambda=\ell^{n}$. Hence, $G_{K_{0}}$ acts on $M$ via a finite quotient $H\subset \GL_{\mathds{F}_{\lambda}}(M)$. Let $L/{K_{0}}$ be the corresponding finite Galois extension and $f_M:S_{L}\longrightarrow S_0$ the induced morphism. We have $H=\Gal(L/K_0)$. Let us denote by $r_M$ the highest jump in the ramification filtration of $H$. Using Herbrand functions \cite[IV 3]{CL}, we will prove that the constant $C_M:=\psi_{L/K_0}(r_M)$ does the job.\\ \indent
Using \ref{purelyins}, we are left to treat the case where $K_0/K$
is separable. The adjunction morphism
$$
 M\longrightarrow f_{M\ast}f^{\ast}_M M
$$
is injective. Since $f^{\ast}_M M\simeq \mathds{F}_{\lambda}^{\rg M}$, we obtain by applying $f_{\ast}$  an injection
$$
f_{\ast} M\longrightarrow \mathds{F}_{\lambda}[\Gal(L/K)]^{\rg M}
$$
So we are left to bound the slopes of $\mathds{F}_{\lambda}[\Gal(L/K)]$ viewed as a $G_K$-representation, that is by \ref{deuxiemefact} \eqref{2} the highest jump in the upper-numbering ramification filtration of $\Gal(L/K)$.
 By \ref{deuxiemefact} \eqref{2},  $r_0:=\Max\Sl_f^{\nb}(\overline{\mathds{Q}}_{\ell})$ is the highest jump in the ramification filtration of $ \Gal(L/K) /H$. Choose $r>\Max(r_0, \varphi_{L/K}\psi_{L/K_0}(r_M))$. We have 
\begin{align*}
\Gal(L/K)^{r} & =  H\cap \Gal(L/K)^{r}\\
   & =   H\cap \Gal(L/K)_{\psi_{L/K}(r)} \\
   &=  H_{\psi_{L/K}(r)}\\
   &= H^{ \varphi_{L/K_0}\psi_{L/K}(r)}\\
   &=\{1\}
\end{align*}
The first equality comes from $r>r_0$. The third equality comes from the compatibility of the lower-numbering ramification filtration with subgroups. The last equality comes from the fact that 
$r>\varphi_{L/K}\psi_{L/K_0}(r_M)$
is equivalent to $\varphi_{L/K_0}\psi_{L/K}(r)>r_M$. Hence,
$$ \Sl_f^{\nb}(M) \subset [0,\Max(r_0, \varphi_{L/K}\psi_{L/K_0}(r_M))]
$$ 
Since $\varphi_{L/K}:[-1,+\infty[\longrightarrow \mathds{R}$ is concave, satisfies $\varphi_{L/K}(0)=0$ and is equal to the identity on $[-1,0]$, we have 
$$
\varphi_{L/K}\psi_{L/K_0}(r_M)\leq \psi_{L/K_0}(r_M)
$$
and we obtain \eqref{quantity} by setting $C_M:=\psi_{L/K_0}(r_M)$.
\end{proof}
\section{Proof of Theorem \ref{theofini}}

\subsection{Preliminary}\label{prelimn}
Let us consider the affine line $\mathds{A}^{1}_S\longrightarrow S$ over $S$. Let $s^{\prime}$ be the generic point of $\mathds{A}^{1}_s$ and $S^{\prime}$ the strict henselianization of $\mathds{A}^{1}_S$ at $s^{\prime}$. We denote by  $\overline{S}$ the normalization of $S$ in $\overline{\eta}$, by $\kappa$ the function field of the strict henselianization of $\mathds{A}^{1}_{\overline{S}}$ at $s^{\prime}$, and by $\overline{\kappa}$ an algebraic closure of $\kappa$. We have $\kappa\simeq K^{\prime}\otimes_K \overline{K}$ and  
\begin{equation}\label{isogal}
G_K\simeq \Gal(\kappa/K^{\prime})
\end{equation}
Let $L/K$ be a finite Galois extension of $K$ in $\overline{K}$. Set $L^{\prime}:=K^{\prime}\otimes_K L$. At finite level, \eqref{isogal} reads
\begin{eqnarray}
\label{iso} \Gal(L/K)& \overset{\sim}{\longrightarrow}    &\Gal(L^{\prime}/K^{\prime}) \\
g &  \longrightarrow     &  \id\otimes g
\end{eqnarray}
Since a uniformizer in $S_L$ is also a uniformizer in $S^{\prime}_{L^{\prime}}$, we deduce that \eqref{iso} is compatible with the lower-numbering ramification filtration on $\Gal(L/K)$ and $\Gal(L^{\prime}/K^{\prime})$. Hence, \eqref{iso} is compatible with the upper-numbering ramification filtration on $\Gal(L/K)$ and $\Gal(L^{\prime}/K^{\prime})$. We deduce that through  \eqref{isogal}, the canonical surjection $G_{K^{\prime}} \longrightarrow G_K$
is compatible with the upper-numbering ramification filtration.
\subsection{The proof}
We can suppose that $\mathcal{F}$ is concentrated in degree 0. In case $\dim X=0$, there is nothing to prove. We first reduce the proof of Theorem \ref{theofini} to the case where $\dim X=1$ by arguing by induction on $\dim X$. \\ \indent
Since the problem is local on $X$, we can suppose that $X$ is affine. We thus have a digram
\begin{equation}\label{factorization}
\xymatrix{
X\ar[r]   \ar[rd]  &\mathds{A}^{n}_S \ar@{^{(}->}[r]   \ar[d]&  \mathds{P}^{n}_S   \ar[ld]\\
 &S & 
}
\end{equation}
Let $\overline{X}$ be the closure of $X$ in $\mathds{P}^{n}_S$ and let $j:X\hookrightarrow  \overline{X}$ be the associated open immersion. Replacing $(X,\mathcal{F} )$ by $(\overline{X}, j_{!}\mathcal{F})$, we can suppose $X/S$ projective. Then Theorem \ref{theofini} is a consequence of the following assertions
\begin{enumerate}
\item[$(A)$] There exists a finite set $E_A\subset \mathds{R}_{\geq 0}$ such that for every $N\in \Sh_c(\eta, \overline{\mathds{Q}}_{\ell})$ with slope not in $E_A$, the support of $\psi_{f}^{t}(\mathcal{F}\otimes f^{\ast}N)$ is punctual.
\item[$(B)$] There exists a finite set $E_B\subset \mathds{R}_{\geq 0}$ such that for every $N\in \Sh_c(\eta, \overline{\mathds{Q}}_{\ell})$ with slope not in $E_B$, we have $$R\Gamma(X_s, \psi_{f}^{t}(\mathcal{F}\otimes f^{\ast}N))\simeq 0$$
\end{enumerate}
Let us prove $(A)$. This is a local statement on $X$, so we can suppose $X$ to be a closed subset in $\mathds{A}^{n}_{S}$ and consider the factorisations
$$
\xymatrix{
X\ar[r]^-{p_i}   \ar[rd]_-{f}  &\mathds{A}^{1}_S\ar[d] \\
& S 
}
$$
where $p_i$ is the projection on the $i$-th factor of $\mathds{A}^{n}_S$. Using the notations in \ref{prelimn}, let $X^{\prime}/S^{\prime}$ making the upper square of the following diagram
$$
\xymatrix{
X^{\prime}\ar[r]^-{\lambda}   \ar[d]_-{p_i^{\prime}}  &X  \ar[d]^-{p_i} \\
S^{\prime}\ar[r] \ar[rd]_-{h}& \mathds{A}^{1}_S \ar[d] \\
                              &     S
}
$$
cartesian. Let us set $\mathcal{F}^{\prime}:=\lambda^{\ast}\mathcal{F}$ and $N^{\prime}:=h^{\ast}N$. From \cite[Th. finitude 3.4]{SGA4undemi}, we have
\begin{equation}\label{lambda}
\lambda^{\ast}\psi_{f}(\mathcal{F}\otimes f^{\ast}N)\simeq \psi_{hp^{\prime}_i}(\mathcal{F}^{\prime}\otimes p^{\prime\ast}_i N^{\prime})\simeq  \psi_{p^{\prime}_i}(\mathcal{F}^{\prime}\otimes p^{\prime\ast}_i N^{\prime})^{G_{\kappa}}
\end{equation}
where $G_{\kappa}$ is a pro-$p$ group sitting in an exact sequence
$$
\xymatrix{
1\ar[r]& G_{\kappa} \ar[r]& G_{K^{\prime}} \ar[r] & G_K \ar[r] & 1 
}
$$
In particular, $G_{\kappa} $ is a subgroup of the wild-ramification group $P_{K^{\prime}}$ of $G_{K^{\prime}}$. So applying the $P_{K^{\prime}}$-invariants on  \eqref{lambda} yields
\begin{equation}\label{doitetrenul}
\lambda^{\ast}\psi_{f}^{t}(\mathcal{F}\otimes f^{\ast}N)   \simeq \psi_{p^{\prime}_i}^{t}(\mathcal{F}^{\prime}\otimes p^{\prime\ast}_i N^{\prime})
\end{equation}
If $N$ has pure slope $r$, we know from \ref{prelimn} that  $N^{\prime}$ has pure slope $r$ as a sheaf on $\eta^{\prime}$. Applying the recursion hypothesis gives a finite set $E_i\subset \mathds{R}_{\geq 0}$ such that the right-hand side of \eqref{doitetrenul} is 0 for $N$ of slope not in $E_i$. The union of the $E_i$ for $1\leq i \leq n$ is the set $E_A$ sought for in $(A)$. \\ \indent
To prove $(B)$, we observe that the compatibility of $\psi^{t}_f$ with proper morphisms and the projection formula give
$$
R\Gamma(X_s, \psi_{f}^{t}(\mathcal{F}\otimes f^{\ast}N))\simeq \psi_{\id}^{t}(Rf_{\ast }\mathcal{F}\otimes N)
$$
By \ref{toutptilemme}, the set $E_B:=\Sl(Rf_{\ast }\mathcal{F})$ has the required properties.\\ \indent
We are thus left to prove Theorem \ref{theofini} in the case where $\dim X=1$. At the cost of localizing, we can suppose that $X$ is local and maps surjectively on $S$. Let $x$ be the closed point of $X$. Note that $k(x)/k(s)$ is of finite type but may not be finite. Choosing a transcendence basis of $k(x)/k(s)$  yields a factorization $X\longrightarrow S^{\prime}\longrightarrow S$ satisfying $\trdeg_{k(s^{\prime})}k(x)=\trdeg_{k(s)}k(x) -1$. \\ \indent
So we can further suppose that $k(x)/k(s)$ is finite. Since $k(s)$ is algebraically closed, we have $k(x)=k(s)$. If $\hat{S}$ denotes the completion of $S$ at $s$, we deduce that $X\times_S \hat{S}$ is finite over $\hat{S}$. By faithfully flat descent \cite[VIII 5.7]{SGA1}, we obtain that $X/S$ is finite. We conclude the proof of Theorem \ref{theofini} with \ref{deuxiemefact} \eqref{1}.


\section{Proof of Theorem \ref{mainth}}
\subsection{}
That $0\in \Sl_{f}(\overline{\mathds{Q}}_{\ell})$ is easy by looking at the smooth locus of $f$. We are left to prove that for every $N\in \Sh_c(\eta, \overline{\mathds{Q}}_{\ell})$ with slope $>0$, the following holds
\begin{equation}\label{annulation}
\psi_f^{t} f^{\ast}N\simeq 0
\end{equation}
Since the problem is local on $X$ for the étale topology, we can suppose that $X=S[x_1,\dots, x_n]/(\pi-x_1^{a_1}\cdots x_m^{a_m})$ and we have to prove \eqref{annulation} at the origin $0\in X_s$. Let $a$ be the lowest commun multiple of the $a_i$ and define $b_i=a/a_i$. Note that $a$ and the $b_{i}$ are prime to $p$. Hence the morphism $h$ defined as
\begin{eqnarray*}
Y:=S[t_1,\dots, t_n]/(\pi-t_1^{a}\cdots t_m^{a}) & \longrightarrow    &  X\\
(t_1, \dots, t_n) &  \longrightarrow   &(t_1^{b_1}, \dots , t_m^{b_m},t_{m+1}, \dots,  t_n) 
\end{eqnarray*}
is finite surjective and finite etale above $\eta$ with Galois group $G$. Set $g=fh$. Then
$$
(\mathcal{H}^{i}\psi_f^{t}f^{\ast}N)_{\overline{0}}\simeq (\mathcal{H}^{i}\psi_g^{t}g^{\ast}N)_{\overline{0}}^{G}
$$
for every $i\geq 0$, so we can suppose $a_1=\cdots = a_m=a$. Since $a$ is prime to $p$, the map of absolute Galois groups induced by $S[\pi^{1/a}]\longrightarrow S$ induces an identification at the level of the ramification groups. By compatibility of nearby cycles with change of trait \cite[Th. finitude 3.7]{SGA4undemi}, we can suppose $a=1$.\\ \indent
Let us now reduce the proof of Theorem \ref{mainth} to the case where $m=1$. We argue by induction on $m$. The case $m=1$ follows from the compatibility of nearby cycles with smooth morphisms. We thus suppose that Theorem \ref{mainth} is true for $m<n$ with all $a_i$ equal to $1$ and prove it for $m+1$ with all $a_i$ equal to $1$. Let $h:\tilde{X}\longrightarrow X$ be the blow-up of $X$ along $x_{m}=x_{m+1}=0$. Define $g:=fh$ and denote by $E$ the exceptional divisor of $\tilde{X}$. Since $h$ induces an isomorphism on the generic fibers, and since $\psi^{t}_f$ is compatible with proper push-forward, we have 
\begin{equation}\label{annulationinfty}
Rh_{\ast}\psi_g^{t} g^{\ast}N \simeq \psi_f^{t} f^{\ast}N\simeq 0 
\end{equation}
By proper base change, \eqref{annulationinfty} gives
\begin{equation}\label{annulation2}
R\Gamma(h^{-1}(0),(\psi_g^{t} g^{\ast}N)_{|h^{-1}(0)} )\simeq 0
\end{equation}
The scheme $\tilde{X}$ is covered by a chart $U$  affine over $S$ given by
$$
S[(u_i)_{1\leq i\leq  n}]/(\pi-u_1\cdots u_{m})
$$
with $E\cap U$ given by $u_m=0$, and a chart $U^{\prime}$ affine over $S$ given by
$$
S[(u_i)_{1\leq i\leq  n}]/(\pi-u_1\cdots u_{m+1} )
$$
with $E\cap U^{\prime}$ given by $u_{m+1}=0$. By recursion hypothesis, $(\psi_g^{t} g^{\ast}N)_{|h^{-1}(0)}$ is a sky-scraper sheaf supported at the origin $0$ of $U^{\prime}$. Hence, \eqref{annulation2} gives
$$
(\psi_g^{t} g^{\ast}N)_{\overline{0}}\simeq 0
$$
This finishes the induction, and thus the proof of Theorem \ref{mainth}.
\subsection{}\label{geoproof}   Let us give a geometric-flavoured proof of 
$$\mathcal{H}^{0}\psi_f^{t} f^{\ast}N\simeq 0$$ in case $X=S[x_1,\dots, x_n]/(\pi-x_1^{a_1}\cdots x_m^{a_m})$. By constructibility \cite[Th. finitude 3.2]{SGA4undemi}, it is enough to work at the level of germs at a geometric point $\overline{x}$ lying over a closed point $x\in X$. \\ \indent
Hence, we have to prove $H^{0}(C, f^{\ast}N)\simeq 0$ for every connected component $C$ of $X_{x,\eta_t}^{\sh}$. For such $C$, denote by $\rho_C: \pi_1(C)\longrightarrow \pi_1(\eta_t)=P_K$ the induced map. Then $H^{0}(C, f^{\ast}N)\simeq N^{\Image \rho_C}$. Since by definition $N^{P_K}=0$, it is enough to prove that $\rho_C$ is surjective. From V 6.9 and  IX 3.4 of \cite{SGA1},  we are left to prove that $C$ is geometrically connected. To do this, we can always replace $X_{x}^{\sh}$ by its formalization
$\hat{X}_x=\Spec R \llbracket \underline{x}  \rrbracket/(\pi- x_1^{a_1}\cdots x_m^{a_m})$. \\ \indent
By hypothesis, $d:=\gcd(a_1, \dots, a_m)$ is prime to $p$, so $\pi$ has a $d$-root in $K_t$. Hence $\hat{X}_{x,\eta_t}$ is a direct union of $d$ copies of 
$$
\Spec K_t\otimes_R R \llbracket \underline{x}  \rrbracket/(\pi^{1/d}- x_1^{a_1^{\prime}}\cdots x_m^{a_m^{\prime}})
$$
where $a_i=da_i^{\prime}$. So we have to prove the following
\begin{lemme}
Let $a_1, \dots, a_m,d \in \mathds{N}^{\ast}$ with $\gcd(a_1, \dots, a_m)=1$. Then 
\begin{equation}
\Spec \overline{K}\otimes_R R \llbracket \underline{x}  \rrbracket/(\pi^{1/d}- x_1^{a_1}\cdots x_m^{a_m})
\end{equation}
is connected. 
\end{lemme}
\begin{proof}
One easily reduces to the case $d=1$. If $R^{\prime}$ is the normalization of $R$ in a Galois extension of $K$ in  $ \overline{K}$, it is enough to prove that 
$\Spec  R^{\prime} \llbracket \underline{x}  \rrbracket/(\pi- x_1^{a_1}\cdots x_m^{a_m})$
is irreducible. If $\pi^{\prime}$ is a uniformizer of $R^{\prime}$, we have $R^{\prime}\simeq k\llbracket \pi^{\prime}\rrbracket$, we write $\pi=P(\pi^{\prime})$ where $P\in k\llbracket X\rrbracket$ and then we are left to prove that $f_{a,P}:=P(\pi^{\prime})- x_1^{a_1}\cdots x_m^{a_m}$ is irreducible in $k\llbracket x_1,\dots, x_n, \pi^{\prime}\rrbracket$. This follows from $\gcd(a_1, \dots, a_m)=1$ via Lypkovski's indecomposability criterion  \cite[2.10]{Lip} for the Newton polyhedron associated to $f_{a,P}$.

\end{proof}

\bibliographystyle{amsalpha}
\bibliography{Saito-Sato-Seminar}

\end{document}